\documentclass[a4paper,10pt]{amsart}
\usepackage{amssymb, amsmath, latexsym, amscd, verbatim, amsthm, enumerate, array}

\newtheorem{theorem}{Theorem}[section]

\newtheorem{corollary}[theorem]{Corollary}
\newtheorem{proposition}[theorem]{Proposition}

\theoremstyle{definition}

\newtheorem{definition}[theorem]{Definition}

\numberwithin{equation}{section}

\newcommand{\tr}{\textnormal{Tr}}

\newcommand{\cyclo}[1][]{\mathbb{Q}(\zeta_{#1})}

\newcommand{\bn}{\mathbb{N}}

\newcommand{\bz}{\mathbb{Z}}
\newcommand{\bq}{\mathbb{Q}}
\newcommand{\br}{\mathbb{R}}

\newcommand{\fo}{\mathcal{O}}

\newcommand{\fok}{\fo_{K}}

\newcommand{\trkq}{\tr_{K/\bq}}

\newcommand{\mr}[1][]{\mathbb{Q}(\zeta_{#1}+\zeta_{#1}^{-1})}
\newcommand{\mi}[1][]{\mathbb{Q}(\zeta_{#1}-\zeta_{#1}^{-1})}

\title[Upper bounds for the Euclidean minima of abelian fields]{Upper bounds for the Euclidean minima of abelian fields}
\author{Eva Bayer-Fluckiger, Piotr Maciak}
\address{Ecole Polytechnique F\'{e}d\'{e}rale de Lausanne\\EPFL--FSB--MATHGEOM--CSAG \\
1015 Lausanne\\
Switzerland}
\email{eva.bayer@epfl.ch}
\email{piotr.maciak@epfl.ch}
\date{\today}

\begin{document}

\begin{abstract}
The aim of this paper is to survey and extend recent results concerning bounds
for the Euclidean minima of algebraic  number fields. In particular, we give upper bounds for the Euclidean minima of abelian fields of prime power conductor.
\end{abstract}
\maketitle

\section{Introduction}
Let $K$ be an algebraic number field, and let $\fok$ be its ring of integers. We denote by  ${\rm N}: K \to \bq$
the absolute value of the norm map. The number field $K$ is said to be {\it Euclidean} (with respect to the norm) if for every $a,b \in \fok$ with $b \not = 0$ there exist
$c, d \in \fok$ such that $a = bc + d$ and ${\rm N}(d) < {\rm N}(b)$.
It is easy to check that $K$ is Euclidean if and only if for every $x \in K$ there exists $c \in \fok$ such that ${\rm N}(x-c) < 1$. This suggests to look at
$$
M(K) = {\rm sup}_{x \in K} {\rm inf}_{c \in \fok} {\rm N}(x-c),
$$
called the {\it Euclidean minimum} of $K$.

The study of Euclidean number fields and Euclidean minima is a classical one. However, little is known about the precise value of $M(K)$ (see
for instance [Le 95]  for a survey, and the tables of Cerri [C 07] for some numerical results). Hence, it is natural to look for {\it upper bounds} for $M(K)$.
This is also a classical topic, for which a survey can be found in [Le 95].

Let $n$ be the degree of $K$ and $D_K$ the absolute value of its discriminant.
It is shown in [BF 06]  that for any number field $K$, we have $M(K) \le 2^{-n}D_K$.
The case of {\it totally real} fields is especially interesting, and has been the
subject matter of several papers. In particular, a conjecture attributed to Minkowski
states that if $K$ is totally real, then $M(K) \le 2^{-n} \sqrt {D_K}.$ This conjecture
is proved for $n \le 8$, cf. [Mc 05], [HGRS 09], [HGRS 11].

Several recent results concern the case of  {\it abelian fields}. In [BFM 13], upper bounds are given for abelian fields of conductor $p^r$, where
$p$ is an odd prime. The present paper complements these results by handling the case of abelian fields of conductor a power of $2$.
In particular, we have

\medskip
\noindent
{\bf Theorem.} {\it If $K$ is totally real of conductor $p^r$, where $p$ is a prime and $r \ge 2$, then $$M(K) \le 2^{-n} \sqrt {D_K}.$$ }

In other words, Minkowski's conjecture holds for such fields.

\medskip

These results are based on the study of lattices associated to number fields (see [BF 99], [BF 02]).  In \S 2, we recall some results on lattices and number fields,
and in \S 3 we survey the results of [BFM 13] concerning abelian fields of conductor an odd prime power. The case of abelian fields of power of $2$ conductor is the
subject matter of \S 4.  Finally, we survey some results concerning cyclotomic fields and their maximal totally real subfields in \S 5.

\section{Lattices and number fields}

We start by recalling some standard notions concerning Euclidean lattices (see for instance
[CS 99] and [M 96].  A {\it lattice} is a pair $(L,q)$, where $L$ is a free $\bz$--module of finite rank, and  $q : L_{\br} \times L_{\br} \to \br$
is a positive definite symmetric bilinear form, where $L_{\br} = L \otimes_\bz \br$.
If $(L,q)$ is a lattice and $a \in \br$, then we denote by $a(L,q)$ the lattice
$(L,aq)$. Two lattices $(L,q)$ and $(L',q')$ are said to be
{\it similar} if and only if there exists $a \in \br$ such that $(L',q')$ and $a(L,q)$ are isomorphic,
in other words if  there exists an isomorphism of $\bz$-modules $f : L \to L'$ such that $q'(f(x),f(y)) = a q(x,y)$.

\medskip

Let $(L,q)$ be a lattice, and set $q(x) = q(x,x)$. The {\it maximum }of $(L,q)$ is defined by
\[
 {\rm max}(L,q) = \sup_{x \in L_{\br}} \inf_{c \in L} q(x-c).
\]
Note that ${\rm max}(L,q)$ is the square of the covering radius of the associated sphere covering.
The {\it determinant} of $(L,q)$ is denoted by ${\rm det}(L,q)$. It is by definition the determinant of the matrix of $q$ in a $\bz$--basis of $L$.
The {\it Hermite--like thickness} of $(L,q)$ is
$$
\tau(L,q) = {{\rm max}(L,q) \over {\rm det}(L,q)^{1/m}},
$$
where $m$ is the rank of $L$. Note that $\tau(L,q)$ only depends on the similarity class of the lattice $(L,q)$.

Next we introduce a family of lattices that naturally occur in connection with abelian fields (see \S 2),
and for which one has good upper bounds of the Hermite-like thickness. This family is defined as follows :

Let $m\in {\mathbb {N}}$, and $b\in {\br}$ with $b>m$.
Let $L=L_{b,m}$ be a lattice in ${\br} ^m$ with Gram matrix
$$b I_m - J_m = \left( \begin{array}{cccc}
b-1 & -1 & \ldots  & -1 \\
-1 & \ddots & \ddots & \vdots  \\
\vdots & \ddots & \ddots & -1  \\
-1 & \ldots & -1 & b-1
\end{array} \right),$$
where $I_m$ is the $m\times m$-identity matrix and $J_m$ is the all-ones matrix of size $m \times m$.
These lattices were defined in [BFN 05],  (4.1).
Note that the lattice $L_{m+1,m} $ is similar to the dual lattice $A_m^{\#}$ of the
root lattice $A_m$ (see for instance [CS 99], Chapter 4, \S 6, or [M 96]  for the definition of the root lattice $A_m$).

Let $K$ be an number field of degree $n$, and suppose that $K$ is either totally
real or totally complex. Let us denote by
$^{\overline {\ }} : K  \to K$ the identity in the first case and the
complex conjugation in the second one, and let $P$ be the set of totally
positive elements of the fixed field of this involution.  Let us denote by ${\rm Tr} : K \to \bq$ the trace
map. For any $\alpha \in P$, set $q_{\alpha}(x,y) = {\rm Tr}(\alpha x \overline y)$ for all $x, y \in K$.
Then $(\fok,q_{\alpha})$ is a lattice.  Set
$$
\tau_{\rm min}(\fok) = {\rm inf} \{ \tau(\fok,q_{\alpha}) \ | \ \alpha \in P \}.
$$

If $D_K$ is the absolute value of the discriminant of $K$, then, by [BF 06], Corollary (5.2), we have
\begin{equation}\label{E:gen-est}
 M(K) \leq \left(\frac{\tau_{\min}(\fok)}{n}\right)^{\frac{n}{2}} \sqrt{D_K},
\end{equation}
This  is used in [BF 06], [BFN 05], [BFS 06] and [BFM 13] to give upper bounds of Euclidean minima (see also \S 3 and \S 5). The bounds
of \S 4 are also based on this result.

\section{Abelian fields of odd prime power conductor}
The set of all abelian extensions of $\bq$ of odd prime power conductor will be denoted by $\mathcal{A}$.
For $K \in \mathcal{A}$ we denote by
$n$  the degree of  $K / \bq$,  by
$D$ the absolute value of the discriminant of  $K$, by $p$
 the unique prime dividing the conductor of  $K$, and by $r$ the $p$-adic additive valuation of the conductor of $K$.
If the dependence on the field $K$ needs to be emphazized, we shall add the index $K$ to the above symbols. For example, we shall write $n_K$ instead of $n$.

\begin{definition}
Let  $\mathcal{D} \subset \mathcal{A}$, and let $\psi : \mathcal{D} \to \mathbb{R}$ be a function. We shall say that
$\psi_o \in \mathbb{R}$ is the limit of $\psi$ as $n_K$
goes to infinity and write
\[
 \lim_{n_K \to \infty} \psi(K)  = \psi_0
\]
if for every $\epsilon > 0$ there exists $N >0$ such that for every field $K \in \mathcal{D}$
\[
 n_K > N  \implies  |\psi(K) - \psi_0| < \epsilon.
\]
We shall also write
\[
 \lim_{p_K \to \infty} \psi(K)  = \psi_0
\]
if for every $\epsilon > 0$ there exists $N >0$ such that for every field $K \in \mathcal{D}$
\[
p_K > N \implies |\psi(K) - \psi_0| < \epsilon.
\]
\end{definition}

The following is proved in [BFM 13], th. (3.1) and (3.2) :

\begin{theorem}\label{T:bound1}
Let $K \in \mathcal{A}$. Then there exist constants $\varepsilon=\varepsilon(K) \leq 2$ and $C=C(K) \leq \frac{1}{3}$ such that
\[
  M(K) \leq  C^n \, (\sqrt{D_K})^{\varepsilon}.
\]
If $[K :  \bq] > 2$, then one may choose $\varepsilon(K) < 2$. Moreover,
\[
 \lim_{n_K \to \infty} \varepsilon(K) = 1.
\]
If $r_K \geq 2$, or $r_K=1$ and $[\cyclo:K]$ is constant, then we also have
\[
 \lim_{p_K \to  \infty} C(K) = \frac{1}{2\sqrt{3}}.
\]
\end{theorem}

\bigskip

\begin{theorem}\label{T:bound2}
Let $K \in \mathcal{A}$. Then there is a constant $\omega=\omega(K)$ such that
\[
 M(K) \leq \omega^n \sqrt{D_K}.
\]
If $r_K \geq 2$, or $r_K=1$ and $[\cyclo:K]$ is constant, then
\[
 \lim_{p_K \to \infty} \omega(K) = \frac{1}{2\sqrt{3}}.
\]
Moreover, if $r_K \geq 2$, then $\omega(K) \leq 3^{-2/3}$.
\end{theorem}

\bigskip

Note that this implies that Minkowski's conjecture holds for all totally real fields
 $K \in \mathcal{A}$ with composite conductor :

 \begin{corollary}\label{T:Minkowski}

 Let $K \in \mathcal{A}$, and suppose that the conductor of $K$ is of the form $p^r$ with
 $r  > 1$. Then
\[
 M(K) \leq 2^{-n} \sqrt{D_K}.
\]

\end{corollary}

This follows from Theorem~(\ref{T:bound2}), since $3^{-2/3} < 1/2$, and for $K$ totally real this is precisely
Minkowski's conjecture.

The proofs of these results are based on the method of [B 05], outlined in the previous section. For $K \in \mathcal{A}$, we denote by
$\fok$ is the ring of integers of $K$, and we consider the
lattice $(\fok,q)$, where $q$ is defined by $q(x,,y) = \trkq(x \overline y)$. As we have
seen in \S 4, the Hermite--like thickness of this lattice can be used to give an upper bound of
the Euclidean minimum of $K$.

Let  $\zeta$ be a primitive root of unity
of order  $p^r$, let us denote by  $e$ the degree
$[\cyclo : K]$. Let $\Gamma_K$  be the orthogonal sum of
$\frac{p^{r-1}-1}{e}$ copies of the lattice $p^{r-1}A_{p-1}^{\#}$.
Set  $d=\frac{p-1}{e}$, and let  $\Lambda_K = e p^{r-1} L_{\frac{p}{e},d}$ (note that
the scaling is taken in the sense of the previous section, that is it refers to multiplying the quadratic form by
the scaling factor). We have (see [BFM 13],  Theorem (6.1)) :

\begin{theorem}\label{T:ortho-sum}
The lattice $(\fok,q)$ is isometric to the orthogonal sum of $\Gamma_K$ and of $\Lambda_K$.
\end{theorem}

This leads to  the following
upper bound
of $\tau_{\rm min}(\fok)$ :

\begin{corollary}\label{c:tau-min}
We have
\[
 \tau_{\rm min}(\fok) \leq  \tau(\fok,q)   \leq n \cdot p^{r-\frac{\upsilon}{n}} \cdot \frac{p^{r+1}+p^r+1-e^2}{12 p^{r+1}},
\]
where
\[
 \upsilon = rn -\frac{(p^{r-1}-1)}{e} -1.
\]
\end{corollary}

This is proved in [BFM 13], Corollary  (6.7). The proof uses an upper bound for the Hermite--like thickness of the lattices $\Lambda_{b,m}$ proved in [BFN 05], (4.1).

\medskip
Using this corollary and (2.1), one proves  (3.2) and (3.3) as in [BFM 13], \S 7.

\section{Abelian fields of power of two conductor}
We keep the notation of \S 1 and 2. In particular, $K$ is a number field of degree $n = n_K$, and the absolute value of its disciriminant is
denoted by $D = D_K$.
Let $r \in \bn$, and let $\zeta$ be a primitive $2^r$-th root of unity. A field $K$
is said to have {\it conductor $2^r$} if $K$ is contained in the cyclotomic field
$\cyclo$, but not in $\bq(\zeta^2)$. Note that there is no field of conductor $2$ and the only field of conductor $4$ is $\bq(i)$. For $r \geq 3$, we have

\begin{proposition}\label{P:bounds}
Suppose that $K$ is an abelian field of conductor $2^r$, where $r \geq 3$. Then we have
$$
K = \cyclo, \mr, \textnormal{ or } \mi.
$$
Moreover,
\begin{enumerate}
 \item[(a)] If  $K = \cyclo$ or $\mr$, then
	    $$
	    M(K) \le 2^{-n} \sqrt {D_K}.
            $$
 \item[(b)] If $K = \mi$, then
	    $$
	    M(K) \le 2^{-n}(2n-1)^{\frac{n}{2}}.
	    $$
\end{enumerate}
\end{proposition}

\begin{proof}
Let us first prove that $K = \cyclo$,  $\bq(\zeta + \zeta^{-1})$  or $\mi$.
Let $G = {\rm Gal}(\cyclo/\bq)$. Then $G = \left<\sigma,\tau\right>$ where $\tau$ is the complex conjugation, and
$\sigma (\zeta) = \zeta^3$. The subgroups of order $2$ of $G$ are
$$
H_1 = \left<\tau\right>, H_2 = \left<\sigma^{2^{r-3}}\right>, \textnormal{ and } H_3 = \left<\sigma^{2^{r-3}} \tau\right>.
$$
It is easy to check that  we have $\cyclo^{H_1} = \mr$, $\cyclo^{H_2} = \bq(\zeta^2)$ and $\cyclo^{H_3} = \mi$.
Note that any proper subfield of $\cyclo$ is a subfield of $\cyclo^{H_i}$ for $i = 1,2$ or $3$, and the statement easily follows  from this observation.

\medskip

Part (a) follows from  Proposition (10.1) in [BF 06] for $K = \cyclo$, and from Corollary (4.3) in [BFN 05] for $K = \bq(\zeta + \zeta^{-1})$.

\medskip

Let us prove (b). Suppose that $K = \mi$.  Recall that $\cyclo^{H_3} = \mi$, and that the ring of integers of $\cyclo$ is $\bz[\zeta]$.
Therefore  the ring of integers of  $\mi$ is $\bz[\zeta -\zeta^{-1}]$. Set $e_i = \zeta + (-1)^i \zeta^{-1}$ for all $i \in \bz$ and $n = 2^{r-2}$. Then
an easy computation shows that the elements $1, e_1, \dots, e_{n-1}$ form an integral basis of $\fok = \bz[\zeta -\zeta^{-1}]$,
and that for $-2n \leq i \leq 2n$ we have
\[
\trkq(e_i) =
\begin{cases}
  2n \quad &\textnormal{ if } i=0,\\
 -2n \quad &\textnormal{ if } i=\pm 2n,\\
 0 \quad &\textnormal{ otherwise. }
\end{cases}
\]

\medskip

Recall that $\tau : K \to K$ is the complex conjugation, and
let
$$
q : \fok \times \fok \to \bz
$$
given by
$$
q(x,y) = \trkq (x \tau(y))
$$
be the trace form. Then the Gram matrix of $(\fok,q)$ with respect to the basis $1,e_1,\dots, e_n$ is
$$
{\rm diag}(n,2n,\dots,2n).
$$
Note that this implies that we have
$$D_K = n^n 2^{n-1}.$$ Since $1,e_1,\dots,e_{n-1}$ is an orthogonal basis for $(\fok,q)$, it
follows from the Pythagorean theorem that the point $x = \frac{1}{2}(1 + e_1 + \dots + e_{n-1})$
is a deep hole of the lattice $(\fok,q)$. Thus we have
$$
\max(\fok,q) = \inf_{c \in \fok} q(x-c) = q(x) = \frac{n(2n-1)}{4}.
$$
Set
$$
\tau(\fok,q) = \frac{\max(\fok,q)}{\det(\fok,q)^{1/n}}.
$$
Then we have
$$
\tau(\fok,q) =  \frac{n (2n-1)}{4(n^n 2^{n-1})^{1/n}} = \sqrt[n]{2}  \cdot \left(\frac{2n - 1}{8}\right).
$$
By (2.1), we have
$$
M(K)  \le  \left(  \frac{\tau(\fok,q)}{n} \right)^{\frac{n}{2}} \sqrt  {D_K}.
$$
Thus we obtain
$$
M(K) \le 2^{-n}(2n-1)^{\frac{n}{2}}.
$$
This completes the proof of the proposition.
\end{proof}

Let $r \geq 3$ and let $K$ be an abelian field of conductor $2^r$ of the form $\mi$. The following two corollaries show an asymptotic behavior of the bound obtained in Proposition~(\ref{P:bounds})(b).

\begin{corollary}\label{C:bound-with-varepsilon}
We have
\[
 M(K) \leq 2^{-n}  (\sqrt{D_K})^{1+\varepsilon(n)},
\]
where
\[
 \varepsilon(n) \sim  \frac{\ln 2 - \frac{1}{2}}{n \ln n}.
\]
\end{corollary}

\begin{proof}
We set
\[
 \varepsilon(n) = \frac{\ln a_n + \ln 2}{n \ln (2n) - \ln 2},
\]
where
\[
a_n = \left(1 - \frac{1}{2n}\right)^n.
\]
Using the fact that $D_K = n^n 2^{n-1}$, the inequality of Proposition~(\ref{P:bounds})(b) can rewritten as
\[
  M(K) \leq 2^{-n}  (\sqrt{D_K})^{1+\varepsilon(n)}.
\]
A simple calculation shows that
\[
 \lim_{n \to \infty} (n \ln n) \cdot \varepsilon(n) = \ln 2 - \frac{1}{2}.
\]
The result follows.
\end{proof}

\begin{corollary} We have
\[
 M(K) \leq (\sqrt{2} \, e^{-1/4}) \cdot 2^{-n}  \sqrt{D_K}.
\]
\end{corollary}

\begin{proof}
Using the same notation as in the proof of Corollary~(\ref{C:bound-with-varepsilon}), we can rewrite the inequality of Proposition~(\ref{P:bounds})(b) as
\[
 M(K) \leq (\sqrt{2 a_n}) \cdot 2^{-n}  \sqrt{D_K}.
\]
The result follows from the fact that the sequence $(a_n)$ is increasing and its limit equals to $e^{-1/2}$.
\end{proof}

\section{Cyclotomic fields and their maximal totally real subfields}
Let $m \in \bn$, and let $\zeta$ be a primitive $m$-th root of unity. Let $K = \cyclo$, and let $F = \mr$ be its maximal totally
real subfield.
Let us denote by $n_K$, respectively $n_F$, their degrees, and by $D_K$, respectively $D_F$, the absolute values of their discriminants.
The aim of this section is to survey some results concerning $M(K$) and $M(F)$.

\begin{theorem} We have
\[
M(K) \leq 2^{-n_K} \sqrt{D_K}
\]
\end{theorem}

\begin{proof} This is proved in [BF 06], Proposition (10.1).

\end{proof}

For certain values of $m$, one obtains better bounds :

\begin{theorem} We have

{\rm (i)} Suppose that $m$ is of the form $m = 2^r3^s5^t$, with $r \ge 0$,
$s \ge 1$ and $t \ge 1$; $m = 2^r 5^s$ with $r \ge 2$, $s \ge 1$;
$m = 2^r 3^s$ with $r \ge 3$, $s \ge 1$. Then
$$M(K) \le 8^{-n_K / 2} \sqrt{D_K}.$$

{\rm (ii)} Suppose that $m$ is of the form $m = 2^r 5^s 7^t$, with
$r \ge 0$, $s \ge 1, t \ge 1$; $m = 2^r 3^s 5^t$ with
$r \ge 0, s \ge 2, t \ge 1$; $m = 2^r 3^s 7^t$ with $r \ge 2, s \ge 1,
t \ge 1$. Then $$M(K) \le 12^{-n_K / 2} \sqrt{D_K}.$$

\end{theorem}

\begin{proof} This is proved in [BF 06], Proposition (10.2). The result follows from (2.1), and the fact that an orthogonal sum of lattices of type $E_8$ is defined over $K$
in the sense of \S 2 in case {\rm (i)}, and an orthogonal sum of lattices isomorphic to the Leech lattice in case {\rm (ii)}.

\end{proof}

The results are less complete for the maximal totally real subfields. We have

\begin{theorem} Suppose that $m = p^r$ where $p$ is a prime and $r \in \bn$, or that $m = 4 k$ with $k \in \bn$ odd. Then we have
\[
M(F) \leq 2^{-n_F} \sqrt{D_F}
\]
\end{theorem}

\begin{proof} This is proved in [BF 06], Proposition (8.4) for $m = p^r$ and $p$ and odd prime; in [BFN 05], Corollary (4.3) for $m = 2^r$; and in [BFS 06], Proposition (4.5) for $m = 4k$.

\end{proof}

\bigskip
{\bf Bibliography}

\medskip

[BF 99] E. Bayer--Fluckiger,  {\it Lattices and number Fields}, Contemp. Math., {\bf 241} (1999),
69--84.

[BF 02] E. Bayer-Fluckiger, {\it Ideal Lattices}, proceedings of the conference Number Theory
and Diophantine Geometry, (Zurich, 1999),  Cambridge Univ. Press
(2002), 168--184.

[BF 06] E.  Bayer-Fluckiger, {\it Upper bounds for Euclidean minima of algebraic number fields},
J. Number Theory {\bf 121} (2006), no. 2, 305-323.

[BFM 13] E. Bayer--Fluckiger, P. Maciak, {\it Upper bounds for the Euclidean minima of
abelian fields of odd prime power conductor}, Math. Ann. (to appear).

[BFN 05] E.  Bayer-Fluckiger, G.  Nebe, {\it On the euclidean minimum of some real number fields},
J. th. nombres Bordeaux, {\bf  17}  (2005), 437-454.

[BFS 06], E. Bayer-Fluckiger, I. Suarez, {\it Ideal lattices over totally real number fields and Euclidean minima},
 Archiv Math.  {\bf86} (2006), 217-225.

[C 07] J.-P. Cerri, {\it Euclidean minima of totally real number fields. Algorithmic determination}, Math. Comp. {\bf 76} (2007), 1547-1575.

[CS 99]  J. H. Conway, N. J. A. Sloane, {\it Sphere Packings, Lattices and Groups}, Third Edition., Springer-Verlag  New York, Inc. (1999).

[HGRS 09]  R. J. Hans-Gill, M. Raka and R. Sehmi, \textit{On Conjectures of Minkowski and Woods for $n = 7$},
J. Number Theory {\bf 129}   (2009), 1011-1033.

[HGRS 11] R. J. Hans-Gill, M. Raka and R. Sehmi, \textit{On Conjectures of Minkowski and Woods for $n = 8$}, Acta Arith. {\bf 147}  (2011), 337-385.

[Le 95]  F. Lemmermeyer, \textit{The Euclidean algorithm in algebraic number fields}, Expo. Math. {\bf 13} (1995), 385-416.

[M 96] J. Martinet, {\it R\'eseaux parfaits des espaces euclidiens}, Masson (1996),
English translation: {\it Perfect lattices of Euclidean spaces}, Springer-Verlag, Grundlehren
der math. Wiss, {\bf 327} (2003).

[Mc 05] C.T. McMullen, \textit{Minkowski's conjecture, well-rounded lattices and topological dimension}, J. Amer. Math. Soc. {\bf 18} (2005) 711-734.

\end{document}